\documentclass[12pt]{amsart}

\usepackage{stmaryrd}

\usepackage{amsfonts,xcolor}
\usepackage{latexsym,amssymb}
\usepackage{amsmath}
\usepackage[utf8]{inputenc}

\newtheorem{theorem}{Theorem}[section]
\newtheorem{lemma}[theorem]{Lemma}
\newtheorem{proposition}[theorem]{Proposition}

\newtheorem{remark}[theorem]{Remark}
\newtheorem{definition}[theorem]{Definition}
\newtheorem{thmx}{Theorem}

\newcommand{\R}{\mathbb{R}}

\newcommand{\N}{\mathbb{N}}
\newcommand{\di}{\, \mathrm{d}}
\newcommand{\csch}{\mathrm{csch}}

\usepackage[colorlinks, linkcolor=blue, citecolor=red, breaklinks]{hyperref}

\catcode`@=11 \@addtoreset{equation}{section} \catcode`@=12

\usepackage[pagewise]{lineno}

\usepackage[left=3cm,right=3cm,top=3cm,bottom=3cm]{geometry}

\begin{document}
	
\title[Heat equation involving a Grushin operator]{On the Heat equation involving a Grushin operator in Marcinkiewicz spaces}

\author[A. E. Kogoj]{Alessia E. Kogoj}
\address{Università degli Studi di Urbino Carlo Bo, Dipartimento di Scienze Pure ed Applicate, Urbino-PU, Italy}
\email{alessia.kogoj@uniurb.it}\thanks{}

\author[M. E. Lima]{Maria E. Lima}
\address{Universidade Estadual de Campinas, IMMEC, Campinas-SP, Brazil}
\email{m227985@dac.unicamp.br}\thanks{}

\author[A. Viana]{Arl\'ucio Viana}
\address{Universidade Federal de Sergipe, Departamento de Matem\'atica, 49100-000 São Cristóvão-SE, Brazil}
\email{arlucioviana@academico.ufs.br}\thanks{*}

\keywords{Subelliptic operators, well-posedness of PDEs, Grushin operator, self-similar solutions}

\subjclass[2020]{35H20, 35B44, 47D06, 35K58 ,35B60, 35B09}

\begin{abstract} 
	
	In this work, we give sufficient conditions for the existence and uniqueness of the heat equation involving the operator	
	$$	
	\Delta_{\mathcal{G}}=\dfrac{1}{2}\left(\Delta_{x}+|x|^2\Delta_{y}\right)
	$$
	in Marcinkiewicz spaces. Furthermore, we provide sufficient conditions for the existence of positive, symmetric, and self-similar solutions. 
	
\end{abstract}

\maketitle

\section{Introduction and main results}

Since Fujita's paper \cite{Fujita}, the semilinear heat equation
\begin{eqnarray}\label{sheat}
	u_t(x,t) = \Delta u(x,t) + |u(x,t)|^{\rho-1}u(x,t),\  \mbox{in}\ (0,\infty)\times \R^N ,\\
	u(x,0)=u_0(x),\ \mbox{in}\ \R^N , \label{sheat0}
\end{eqnarray} 
$\rho>1$, is of great mathematical interest.

It is well-known that if $u_0\in C_0(\R^N)$ be nonnegative and nonzero, then
\begin{enumerate}
	\item If $1<\rho<1+\frac{2}{ N}$, there exists no positive global solution of \eqref{sheat}--\eqref{sheat0}.
	\item If $\rho> 1+\frac{2}{ N}$, there exists $u_0\in L^{\frac{ N}{2}(\rho-1)}(\R^N)$ such that there exists a global positive solution of \eqref{sheat}--\eqref{sheat0}.
\end{enumerate}
The critical case $\rho = 1+\frac{2}{N}$ was resolved by Weissler \cite{Weissler}. Initial conditions in Lebesgue spaces were considered in \cite{BrezisCaz,Weissler}. 

In this work, we will replace the classical Laplacian with the operator 
\begin{equation}\label{Grushiop}
	\Delta_{\mathcal{G}}=\dfrac{1}{2}\left(\Delta_{x}+|x|^2\Delta_{y}\right), 
\end{equation}
where $\Delta_x, \Delta_y$ denote the classical Laplacian in the variables $x\in\R^N$ and $y\in\R^k$, respectively.

This operator, nowadays called Grushin-type (see \cite{Grushin, grushin_1971}), actually belongs to the general class of operators studied by H\"ormander in \cite{hormander} and is hypoelliptic. Moreover, it is also a particular case of the degenerate elliptic operators studied in  \cite{Kogoj-Lan-12} by Kogoj and Lanconelli.  The first  to introduce and study a metric and a consequent underlying sub-Riemannian structure for such operators were Franchi and Lanconelli in their seminal papers in the early 1980s 
\cite{franchi_lanconelli_1, franchi_lanconelli_2, franchi_lanconelli_3}.
In recent years, many works have appeared in the literature dealing with second-order linear and semilinear degenerate elliptic PDOs, which fall into this class. We refer to  \cite{nina, Ba-Furutani-15, D-S-Zhu-22,Kogoj-Lan-18,La-Lu-Mu-21} and the references therein for recent results involving the Grushin-type operators.

We are dealing with the following problem 

\begin{eqnarray}\label{sheatG}
	u_t = \Delta_\mathcal{G} u + |u|^{\rho-1}u,\  \mbox{in}\ (0,\infty)\times \R^{N+k} ,\\ \nonumber
	u(0)=u_0,\ \mbox{in}\ \R^{N+k}. \label{sheatG0}
\end{eqnarray}

In bounded domains, we can find well-posedness, long-time dynamics, and the existence of attractors for semilinear equations involving degenerate elliptic operators containing the Grushin ones in \cite{Kogoj-Sonner-13,Kogoj-Sonner-14,Liu-Tian-24}. In particular, the authors of \cite{Liu-Tian-24} applied the Galerkin method to prove the existence of solutions of parabolic and pseudo-parabolic equations associated with Hörmander-type operators, where they used Sobolev-type estimates, which are available for bounded domains. They also proved the existence and upper continuity of attractors.

In \cite{Oli-Vi-23}, Oliveira and Viana prove existence, uniqueness, continuous dependence and blowup alternative of local mild solutions for \eqref{sheatG}--\eqref{sheatG0} with initial conditions in Lebesgue spaces. Also, they obtain the existence of global solutions in the special case of $u_0\in L^{\frac{ N}{2}(\rho-1)}(\R^{N+k})$ with a sufficiently small norm. 

%
In this manuscript, by working in Marcinkiewicz spaces, we allow larger (in Lebesgue norm) initial data to be taken into account to obtain global solutions. More precisely, we give sufficient conditions for the existence and uniqueness of mild solutions for \eqref{sheatG}--\eqref{sheatG0}, with initial conditions in the critical Marcinkiewicz space $L^{(p,\infty)}(\R^{N+k})$, with 
\begin{equation}\label{p}
	p = \frac{N+2k}{2}(\rho-1) .
\end{equation}
Then, we prove the existence of positive, symmetric and self-similar solutions. Indeed, for example, if $u_0(x,y) = \varepsilon|x|^{-\frac{2}{\rho-1}} |y|^{-\frac{1}{\rho-1}}$ and $\varepsilon>0$ is sufficiently small, Theorems \ref{globale} (a) and \ref{self} gives the existence of a positive, self-similar and symmetric solution of \eqref{sheatG}--\eqref{sheatG0} in $L^\infty((0,\infty);L^{(p,\infty)})$. Notice that such an initial condition has an infinity $L^p(\R^{N+k})$-norm, so we cannot apply those results in \cite{Oli-Vi-23}. In other words, Marcinkiewicz spaces allow singular homogeneous initial conditions that generate positive, self-similar solutions. Moreover, Theorem \ref{globale} also gives sufficient conditions for uniqueness, regularity and time-decay of the mild solutions.

The number $N+2k$ is the so-called homogeneous dimension attached to the operator in \eqref{Grushiop}  (see \cite{Kogoj-Lan-18}). 

Here, we will rely on the \textit{explicit} expression of the heat kernel of the Grushin operator (see e.g. \cite{Oli-Vi-23, Garofalo-Trallli-22}):
\begin{equation}\label{HK}
	K(x, x_0,y;t) = \frac{1}{(2\pi)^{\frac{N+2k}{2}}} \int_{\R^k} \left(\frac{|\xi|}{\sinh(|\xi| t)}\right)^{\frac{N}{2}} e^{i\xi  \cdot y- \frac{|\xi|}{2}\left((|x|^2 + |x_0|^2)  \coth(|\xi| t) -2x\cdot x_0 \csch(|\xi|t) \right)} \di\xi, 
\end{equation}
for $(x, x_0,y) \in \R^{2N+k}, t>0.$

To prove this, the authors in \cite{Oli-Vi-23} followed the Geometric Method in \cite{Calin-11} to find the heat kernel of the heat equation with quadratic potential and then applied the partial Fourier transform on the variable $y$. Then, the inverse partial Fourier transform gave \eqref{HK}, an expression that also appears in Theorem 3.4 of the paper by Garofalo and Tralli \cite{Garofalo-Trallli-22}, with slight differences due to constant choices in the definition of the operator. A stochastic approach to study the kernel's short-time asymptotics in two-dimensions was given by Sowers \cite{sowers}. More recently, Stempak \cite{Stempak25} obtained that closed form for the kernel by applying the functional calculus of the Grushin operator and using the Hankel transform.

Let us recall some properties of the heat kernel and the heat semigroup in $L^p$ proved in \cite{Oli-Vi-23}. The kernel \eqref{HK} is $C^\infty$ and 
\begin{enumerate}
	\item $K_t = \Delta_\mathcal{G} K$;
	\item $\int_{\R^{N+k}} K(x,0,y;t) \di (x,y) =1$ ;
	\item $\lim_{t\rightarrow0^+}\int_{\R^{N+k}} K(x,w,y-z;t) \varphi(w,z) \di(w,z) = \varphi(x,y)$ .
\end{enumerate}

Moreover,
\begin{equation}\label{symmetry}
	K(x,x_0,y;t)= t^{-\frac{N+2k}{2}} K(t^{-\frac{1}{2}}x, t^{-\frac{1}{2}}x_0, t^{-1}y;1).
\end{equation} 
and the heat kernel \eqref{HK} is positive. In particular, $\|K(\cdot,0,\cdot;t)\|_{L^1(\R^(N+k)} = 1$.

The unique solution of the Cauchy problem
\begin{eqnarray}\label{cauchy}
	\partial_{t}u-\Delta_{\mathcal{G}}u=0, & x\in \R^{N+k}, \ t>0 \\
	\hspace*{-1cm}u(x,y,0)=u_0(x,y), & x\in \R^{N+k} ,
\end{eqnarray}
defines a strongly continuous semigroup in $L^p$. It is defined by
\begin{equation}\label{solinear}
	S_\mathcal{G}(t)\varphi(x,y) =	\int_{\R^{N+k}} K(x,w,y-z;t) \varphi(w,z) \di(w,z) .
\end{equation} 
Indeed, we have the following result proved in \cite{Oli-Vi-23}.

\begin{thmx}\label{lpest}
	For all $1\leq p \leq \infty$, $S_\mathcal{G}(t):L^p(\R^{N+k})\rightarrow L^p(\R^{N+k})$ is a semigroup. If $1\leq p \leq r\leq\infty$, then
	\begin{equation}\label{semigroup}
		\|S_\mathcal{G}(t)\varphi\|_{L^r(\R^{N+k})}\leq C\|\varphi\|_{L^p(\R^{N+k})}t^{-\frac{N+2k}{2}\left(\frac{1}{p}-\frac{1}{r}\right)};
	\end{equation}
	Furthermore, for $1 \leq p \leq r \leq \infty$, for all $t_0>0$, it is strongly continuous, that is, for $\varphi \in L^r(\R^{N+k})$, it holds
	\begin{equation}\label{continuidadeforte} 
		\|S_\mathcal{G}(t)\varphi-S_\mathcal{G}(t_0)\varphi\|_{L^p}\rightarrow 0,
	\end{equation}
	as $t\rightarrow t_0$. When $p=r<\infty$, then we may take $t_0=0$.
\end{thmx}

We mention that Metafune, Negro and Spina \cite{Me-Ne-Sp-20} have proved $L^p$ estimates for the Grushin operator and used the semigroup associated with it defined in a Sobolev-type space.

Now, we state our results. In the following, $X = L^\infty((0,\infty);L^{(p,\infty)}(\R^{N+k}))$ and a mild solution is a solution in $X$ that satisfies the integral equation associated with \eqref{sheatG}--\eqref{sheatG0} by means of the Duhamel principle. We refer to  Section \ref{main} for the precise definition of the mild solution. 

\begin{theorem}\label{globale}
	Let $1<p<\infty$ be given by \eqref{p}.
	\begin{enumerate}
		\item[(a)] \textbf{Well-posedness.} There exists $\delta_p>0$ sufficiently small such that , if $\|u_0\|_{(p,\infty)}<\delta_p$, then  the problem \eqref{sheatG}-\eqref{sheatG0} has a global mild solution $u \in X$ that is unique in $B_X(2\varepsilon)$ ($\varepsilon>0$ will be precise later). Furthermore, if $u,v\in$ are mild solutions of \eqref{sheatG}--\eqref{sheatG0} with initial conditions $u_0, v_0$ with $L^{(p,\infty)}$ norm less than $\delta_p$, respectively, then 
		\begin{equation}\label{contdep}
			\|u(t) - v(t)\|_{(p,\infty)} \leq \frac{1}{1-2^\rho K_p \varepsilon^{\rho-1}} \|u_0 - v_0\|_{(p,\infty)} 
		\end{equation}
		
		\item[(b)] \textbf{Regularity.} If $u_0\in L^{(p,\infty)}(\R^{N+k}) \cap L^{(q,\infty)}(\R^{N+k})$, $\|u_0\|_{(q,\infty)}<\delta_q$, with $\delta_q>0$ sufficiently small, for $q>\frac{N+2k}{N+2(k-1)}$, then $u \in L^\infty((0,\infty); L^{(p,\infty)}(\R^{N+k}) \cap L^{(q,\infty)}(\R^{N+k}))$. Moreover, for $p<r<q$, we have $u, \in L^\infty(0,\infty;L^r(\R^{N+k}))$.
		
		\item[(c)] \textbf{Decay.} Moreover, if $\rho p < r< q$, we have
		\begin{equation}\label{decay}
			\|u(t)\|_{(r,\infty)}\leq C t^{-\sigma}, \ t>0 ,
		\end{equation} 
		where $\sigma=\frac{N+2k}{2}\left(\frac{1}{p}-\frac{1}{r}\right)$.
		\item[(d)] \textbf{Uniqueness.} The solution in unique in the set $L^\infty((0,\infty); L^{(p,\infty)}(\R^{N+k}) \cap L^{(q,\infty)}(\R^{N+k}))$, for $1<p<q<\infty$.
	\end{enumerate}
\end{theorem}

One key ingredient for the proof of Theorem \ref{globale} is a Yamazaki-type inequality (\cite{Yamazaki}): if $1<p<q<\infty$, there is a constant $C>0$ such that
\begin{equation}
	\int^{\infty}_0t^{\frac{n}{2}\left(\frac{1}{p}-\frac{1}{q}\right)-1}\|G(t)\phi\|_{(q,1)}ds\leq C\|\phi\|_{(p,1)}, 
\end{equation}
for each $\phi \in L^{(p,1)}(\mathbb{R}^n)$, in which $G(t)$ denotes the heat semigroup. Also, it was generalized by Ferreira and Villamizar-Roa (see \cite{FVR}) for the semigroup generated by the fractional Laplacian. We will need to use a similar inequality for the semigroup generated by the Grushin operator (see Lemma \ref{yam} below). The regularity relies on interpolation arguments and the smoothness of the kernel.

In the next theorem, we will use the following terminology. By a \textit{self-similar} solution, we mean a solution $u$ that satisfies the scaling map
\begin{equation}\label{scalingmap}
	u_{\lambda}(x,y,t):=\lambda^{\frac{2}{\rho-1}}u(\lambda x, \lambda^2 y, \lambda t).
\end{equation}
Let $\mathcal{A}$ be a subset of the orthogonal matrices in $\mathcal{O}(N)\times \mathcal{O}(k)$. We say that a function $\varphi$ is invariant under the action of $\mathcal{A}$ if $\varphi(Tz) = \varphi(z)$, for all $z\in\R^{N+k}$ and $T\in \mathcal{A}$.

\begin{theorem}\label{self}
	Let the assumptions of Theorem \ref{globale} hold. 
	\begin{enumerate}
		\item[(a)] If $u_0$ is a nonnegative function, then $u$ is positive.
		\item[(b)] If $u_0$ is a symmetric function under the action of $\mathcal{A}$, then $u$ also is.
		\item[(c)] If $u_0$ is a homogeneous function $-\frac{2}{\rho-1}$, then the solution $u$ given by Theorem \ref{globale} is self-similar.
	\end{enumerate} 
\end{theorem}

The rest of the manuscript is organized as follows. In Section \ref{pre}, we define and gather some properties of the Lorentz and Marcinkiewicz spaces. We also study properties of the semigroup associated with the Grushin operator in Lorentz spaces and the Yamazaki-type estimate. Other key estimates are proved in this section. Section \ref{main} is devoted to the proof of Theorem \ref{globale}, and Section \ref{proofsym} to the proof of the symmetries. We close the paper with some remarks on the solutions to our problem.

\section{Key results}\label{pre}

Most of the notations we use in this paper are standard. Lorentz and Marcinkiewicz spaces $L^{(p,q)}$ and $L^{(p,\infty)}$ are as defined in \cite{bergh, FVR}. The norm in these spaces will be denoted by $\| \cdot \|_{(p,q)}$. Next, we collect the main properties of these spaces. 

Consider a measurable function $f:\mathbb{R}^n\rightarrow \mathbb{R}$.  The $f^{*}$ rearrangement function is defined by
\begin{equation*}\label{rearrang}
	f^{*}(t)=\inf \left\{s>0: m\left(\left\{x \in \mathbb{R}^{n}:|f(x)|>s\right\}\right) \leq t\right\}, \ t>0,
\end{equation*}
where $m$ is the measure $\mathbb{R}^n$-Lebesgue. The averaging function $f^{**}$ is defined by
\begin{equation*}
	f^{* *}(t)=\frac{1}{t} \int_{0}^{t} f^{*}(s) d s, \  t>0.
\end{equation*}
Let $0<p \leqslant \infty, 0<q \leqslant \infty$. The Lorentz space, $L^{(p, q)} (\mathbb{R}^n)$, is the set of all measurable functions $f: \mathbb{R}^{n} \longrightarrow \mathbb{ R}$, such that $\|f\|_{(p, q)}^{*}<\infty$, where
\begin{equation}
	\|f\|^{*}_{(p, q)}=\left\{\begin{array}{ll}
		{\left[\frac{q}{p}\int_{0}^{\infty}\left(t^{\frac{1}{p}} f^{*}(t)\right)^{q} \frac{d t}{t}\right]^{\frac{1}{q}}, \ t>0} & \text { if } 0<p<\infty, 0<q<\infty, \\
		\sup _{t>0} t^{\frac{1}{p}} f^{*}(t), & \text { if } 0<p \leqslant \infty, q=\infty.\\
		
	\end{array}\right.
\end{equation}
Since $\|f\|^{*}_{(p, q)}$ does not satisfy the triangular inequality, $L^{(p,q)}$ is metrizable with the norm $\|f\| _{(p, q)}$ given by
\begin{equation}\label{norma}
	\|f\|_{(p, q)}=\left\{
	\begin{array}{ll}
		\left(\frac{q}{p} \int_{0}^{\infty}\left[t^{\frac{1}{p}} f^{* *}(t)\right]^{q} d t / t \right)^{\frac{1}{q}}, & \text { if } \  1<p<\infty, \ 1 \leq q<\infty; \\
		
		\sup _{t>0} t^{\frac{1}{p}} f^{* *}(t), & \text { if } \  1<p \leq \infty, \ q=\infty.
	\end{array}
	\right.
\end{equation}
When $q=\infty$, the space $L^{(p,\infty)}$ is called the Marcinkiewics space or weak-$L^p$.

\begin{proposition} \label{dual}
	The spaces $L^{(p,q)}$ with the norms $\|f\|_{(p,q)}$ are Banach spaces and
	\begin{equation}\label{h}
		\|f\|^*_{(p,q)}\leq \|f\|_{(p,q)}\leq\frac{p}{p-1}\|f\|^*_{(p,q)}, 
	\end{equation}
	where $1<p\leq \infty$ and $1\leq q \leq \infty$. If, in addition, $(\mathbb{R}^n,\mu)$ be a $\sigma-$finite space, then
	\begin{eqnarray*}
		(L^{(p,q)})^{\ast}&=&L^{(p,\infty)}, \qquad 1<p<\infty, \ 0<q\leq 1,\\
		(L^{(p,q)})^{\ast}&=&L^{(p^{\prime},q^{\prime})}, \qquad 1<p<\infty, \ 0<q< \infty.
	\end{eqnarray*}
	
	\noindent Moreover, given $T\in (L^{(p,q)})^{\ast}$, with $1<p<\infty$ and $0<q\leq\infty$, exists $g$ a measurable function such that 
	\begin{equation}\label{eqdual}
		T(f)=\int_{\mathbb{R}^n}fgd{\mu}.
	\end{equation}
\end{proposition}

The following remark is a consequence of the equivalence between the norm and the seminorm of the Lorentz space.
\begin{remark}\label{ob}
	Let $1<p<\infty$ and $0<q<\infty$. If $h\in L^{(p,\infty)}(\mathbb{R}^n)$, then
	$\||h|^q\|_{(p,\infty)}\leq \frac{p}{p-1} \|h\|^q_{(pq,\infty)}$. 
\end{remark}  
Plus, we need a Hölder-type estimate.
\begin{proposition}[Generalized Hölder inequality]\label{holder}
	Let $1<p_{1}, p_{2}<\infty$. Let $f \in$
	$L^{\left(p_{1}, d_{1}\right)}$, $g \in L^{\left(p_{2}, d_{2}\right)}$, and $\frac{1}{p_{1}}+\frac{1}{p_{2}}<1$. Then the product  $h=f g$ belongs to  $L^{\left(r, d_{3}\right)}$, where $\frac{1}{r}=\frac{1}{p_{1}}+\frac{1}{p_{2}}$, and
	$d_{3} \geq 1$ is such that $\frac{1}{d_{1}}+\frac{1}{d_{2}} \geq \frac{1}{d_{3}} .$ Moreover,
	$$\|h\|_{\left(r, d_{3}\right)} \leq C(r)\|f\|_{\left(p_{1}, d_{1}\right)}\|g\|_{\left(p_{2}, d_{2}\right)}.$$
	If $r=d_3=1$, then $h\in L^1$  and
	$$\|h\|_1\leq\|f\|_{(p,q_1)}\|g\|_{(p^{\prime},q_2)},$$
	where $p^{\prime}$ is the conjugate of $p$.
\end{proposition}

We can use the interpolation theory in \cite{bergh} and Theorem \ref{lpest} to obtain the behaviour of the semigroup generated by the Grushin operator in Lorentz spaces.

\begin{proposition}\label{lpwest}
	For all $1\leq p \leq \infty$, the semigroup $S_\mathcal{G}(t):L^{(p,s)}(\R^{N+k}) \to L^{(r,s)}(\R^{N+k})$ is a strongly continuous semigroup for $t>0$, provided that $1\leq p \leq r\leq\infty$ and $1<s\leq \infty$. Moreover,
	\begin{equation}\label{semigroupw}
		\|S_\mathcal{G}(t)\varphi\|_{(r,s)}\leq C\|\varphi\|_{(p,s)}t^{-\frac{N+2k}{2}\left(\frac{1}{p}-\frac{1}{r}\right)} .
	\end{equation}
	Moreover, given $1<p<\infty, v \in L^{(p, \infty)}$ and $\varphi \in L^{(p', 1)}$, we have
	\begin{equation}\label{duality}
		\langle S_\mathcal{G}(t) v, \varphi\rangle = \langle v,  S_\mathcal{G}(t) \varphi\rangle .
	\end{equation}
\end{proposition}
\begin{proof}
	Recall that Lorentz spaces are real interpolation of Lebesgue spaces, that is, $L^{p,q} = (L^{p_0}, L^{p_1})_{\eta,q}$ with $\frac{1}{p} = \frac{1-\eta}{p_0} + \frac{\eta}{p_1}$. Also, Th. \ref{lpest} gives that $S_\mathcal{G}: L^{p_i}(\R^{N+k}) \to L^{r_i}(\R^{N+k})$, $1\leq p_i\leq r_i\leq \infty$, $i=0,1$. From the Marcinkiewicz interpolation theorem (see \cite{bergh}) and Theorem \ref{lpest}, we obtain that 
	\begin{align*}
		\|S_\mathcal{G}(t)\|_{L^{(p,q)}\to L^{(r,q)}} \leq & c \|S_\mathcal{G}(t)\|_{L^{p_0}\to L^{r_0}}^{1-\eta} \|S_\mathcal{G}(t)\|_{L^{p_1}\to L^{r_1}}^\eta  \\ 
		\leq & C \left(t^{-\frac{N+2k}{2}\left(\frac{1}{p_0}-\frac{1}{r_0}\right)}\right)^{1-\eta} 
		\left(t^{-\frac{N+2k}{2}\left(\frac{1}{p_1}-\frac{1}{r_1}\right)}\right)^{\eta} \\
		= & C t^{-\frac{N+2k}{2}\left(\frac{1}{p}-\frac{1}{r}\right)} ,
	\end{align*}
	because $\frac{1}{p} = \frac{1-\eta}{p_0} + \frac{\eta}{p_1}$ and $\frac{1}{r} = \frac{1-\eta}{r_0} + \frac{\eta}{r_1}$.
	
	Identity \eqref{duality} comes from a combination of the duality relation $(L^{(p, \infty)})^* = L^{(p', 1)}$ and Fubini's theorem.
\end{proof}

Now, let us prove a Yamazaki-type lemma (see \cite{Yamazaki}).
\begin{lemma}\label{yam}
	Let $1<p<r<\infty$. Then,
	\begin{equation}\label{yamineq}
		\int_{0}^{\infty} t^{\frac{N+2 K}{2}\left(\frac{1}{p}-\frac{1}{r}\right)-1} \|S_\mathcal{G}(t) \phi\|_{(r, 1)} d s \leq c\|\phi\|_{(p, 1)}
	\end{equation}
\end{lemma}
\begin{proof}
	Set $\xi(t)=t^{\frac{N+2k}{2}\left(\frac{1}{p}-\frac{1}{r}\right)-1} \left\|S_{\mathcal{G}}(t) \phi\right\|_{(r, 1)}$ and $1<p_{1}<p<p_{2}<r$ such that	
	$$
	\frac{1}{p}=\frac{\lambda}{p_{1}}+\frac{1-\lambda}{p_{2}}, \lambda \in(0,1) .
	$$
	From Proposition \ref{lpwest}, we have	
	$$
	\left\|S_{\mathcal{G}}(t) \phi\right\|_{(r, 1)} \leq c t^{-\frac{N+2 k}{2}\left(\frac{1}{p_{j}}-\frac{1}{r}\right)} \|\phi\|_{(p_j,1)} ,\ \phi\in ^{(p_j,1)}(\R^{N+k}), \ j=1,2.
	$$
	Now, define $\frac{1}{l_j} = \frac{N+2K}{2} \left(\frac{1}{p_j}-\frac{1}{p}\right) + 1$. The above estimate gives us
	\begin{equation*}
		\xi(t) \leq c \ t^{l_j^{-1}} \|\phi\|_{(p_j,1)} .
	\end{equation*}
	Since $\|t^{l_j^{-1}}\|_{(p_j,\infty)}=1$, we conclude that $\|\xi(t)\|_{(l_j,\infty)} \leq c \|\phi\|_{(p_j,1)}$. Note that $1=\frac{\lambda}{l_{1}}+\frac{1-\lambda}{l_{2}}$, $L^{(p, 1)}(\R^{N+k})=\left(L^{\left(p_{1}, 1\right)}(\R^{N+k}), L^{\left(p_{2}, 1\right)}(\R^{N+k})\right)_{\lambda, 1}$ and $L^{(1,1)}(0, \infty)=\left(L^{\left(l_{1}, 0\right)}(0, \infty),L^{\left(l_{2}, 0\right)}(0, \infty)\right)_{\lambda, 1}$. Then, real interpolation again gives us $\xi(t): L^{(p, 1)}(\R^{N+k}) \to L^{(1,1)}(0, \infty)$ and 
	$$ \|\xi(t)\|_{(1,\infty)} \leq c \|\phi\|_{(p,1)} ,$$
	that is \eqref{yamineq}.
\end{proof}

For $f(w)=|w|^{\rho-1}w, w\in\R$, $1<\rho<r<\infty$, we define
$$
\zeta(h)(x)=\int_{0}^{\infty} S_\mathcal{G}(s)(f(h))(s) \ ds.
$$
We have the following lemma.
\begin{lemma}\label{nonl}
	Let $h \in L^{\infty}\left((0, \infty) ; L^{(p, \infty)}\right), p=\frac{N+2 k}{2}(\rho-1)$, and $\rho> \frac{N+2k}{N+2(k-1)}$. Then,
	\begin{equation}
		\|\zeta(h)\|_{(p, \infty)} \leq K \sup_{t>0}\|h(t)\|_{(p, \infty)}^{\rho}	.
	\end{equation}
	
	If, in addition, $h \in L^{\infty}\left((0, \infty) ; L^{(p, \infty)} \cap L^{(q, \infty)}\right)$, for $q>\frac{N+2k}{N+2(k-1)}$, then
	\begin{equation}\label{nonl2}
		\|\mathcal{C}(h)\|_{(q,\infty)} \leq C\sup_{t>0} \|h(t)\|^{\rho -1}_{(p,\infty)}\sup_{t>0} \|h(t)\|_{(q,\infty)} .
	\end{equation}
	
\end{lemma}

\begin{proof}
	Let $\phi \in L^{(p, \infty)}$, note that $\frac{p}{p-\rho}>p'=\frac{p}{p-1}$ and $\frac{N+2 k}{2}\left(\frac{1}{p'}-\frac{p-\rho}{p}\right)-1=0$. From \eqref{duality}, the Hölder inequality \ref{holder}, and the Yamazaki-type inequality \eqref{yamineq}, we have
	\begin{align*}
		\|\zeta(h)\|_{(p, \infty)} = & \sup_{\|\phi\|_{(p', 1)}=1} \left| \int_{\mathbb{R}^{N+K}} \zeta(h) \phi \right|\\
		= & \sup_{\|\phi\|_{(p', 1)}=1} \left| \int_{0}^{\infty} \int_{\mathbb{R}^{N+k}} (S_\mathcal{G}(s) f(h))(s) \phi(z) \ dz \ ds\right| \\
		= &  \sup_{\|\phi\|_{(p', 1)}=1} \left| \int_{0}^{\infty} \int_{\mathbb{R}^{N+k}} f(h)(S_\mathcal{G}(s) \phi) \ dz \ ds\right| \\
		\leq & \sup_{\|\phi\|_{(p', 1)}=1} \int_{0}^{\infty}\|f(h)\|_{\left(\frac{p}{\rho}, \infty\right)} \|S_\mathcal{G}(s) \phi\|_{\left(\frac{p}{p - \rho}, 1\right)} \ ds \\
		\leq & C \sup_{\|\phi\|_{(p', 1)}=1} \int_{0}^{\infty} t^{\frac{N+2 k}{2}\left(\frac{1}{p'}-\frac{p-\rho}{p}\right)-1} \|S_\mathcal{G}(s) \phi\|_{\left(\frac{p}{p - \rho}, 1\right)} ds \left(\sup_{t>0} \|h(t)\|^\rho_{(p,\infty)}\right) \\
		\leq & C \sup_{t>0} \|h(t)\|^\rho_{(p,\infty)} .
	\end{align*}
	This proves \eqref{nonl}.
	
	To prove \eqref{nonl2}, we repeat the idea above, rather than using that
	\begin{eqnarray*}
		\|\mathcal{C}(h)\|_{(q,\infty)}\leq \sup _{\|\phi\|_{(p^{\prime},1)}=1}\int_0^{\infty}\|f(h)\|_{(r,\infty)}\|S_{\mathcal{G}}(t)\phi\|_{(r^{\prime},1)}dt
	\end{eqnarray*}
	and
	\begin{eqnarray*}
		\|f(h)\|_{(r,\infty)}\leq\|h\|^{\rho-1}_{(p,\infty)}\|h\|_{(q,\infty)}
	\end{eqnarray*}
	for $\frac{1}{r}=\frac{\rho-1}{p}+\frac{1}{q}$. Now, we apply Yamazaki' inequality with $(p,q)=(q',r')$ to obtain
	\begin{eqnarray*}
		\|\mathcal{C}(h)\|_{(q,\infty)}\leq C\sup_{t>0} \|h(t)\|^{\rho -1}_{(p,\infty)}\sup_{t>0} \|h(t)\|_{(q,\infty)} ,
	\end{eqnarray*}
	since $q'<r'$, and $\frac{N+2K}{2}\left(\frac{1}{q'}-\frac{1}{r'}\right)=\frac{N+2K}{2}\left(\frac{1}{r}-\frac{1}{q}\right)=\frac{N+2K}{2}\frac{(\rho-1)}{p}=1$.
\end{proof}

Let us recall the following contraction principle based on \cite[Lemma 3.9]{FVR}. It is especially useful to control the size of the constants that will appear in some proofs.

\begin{lemma}\label{fixedpoint}
	Let $1<\rho<\infty$ and $X$ be a Banach space with norm $\|\cdot\|$, and $B: X \rightarrow X$ be a map witch satisfies
	$$
	\|B(x)\| \leq K\|x\|^\rho
	$$
	and
	$$
	\|B(x)-B(z)\| \leq K\|x-z\|\left(\|x\|^{\rho-1}+\|z\|^{\rho-1}\right) .
	$$
	
	Let $R>0$ be the unique positive root of the equation $2^\rho K a^{\rho-1}-1=0$. Given $0<\varepsilon<R$ and $y \in X, y \neq 0$, such that $\|y\|<\varepsilon$, there exists a solution $x \in X$ for the equation $x=y+B(x)$ such that $\|x\| \leq 2 \varepsilon$. The solution $x$ is unique in the ball $B_{2 \varepsilon}:=\bar{B}(0,2 \varepsilon)$. Moreover, the solution depends continuously on $y$ in the following sense: If $\|\tilde{y}\| \leq \varepsilon, \tilde{x}=\tilde{y}+B(\tilde{x})$, and $\|\tilde{x}\| \leq 2 \varepsilon$, then
	$$
	\|x-\tilde{x}\| \leq \frac{1}{1-2^\rho K \varepsilon^{\rho-1}}\|y-\tilde{y}\| .
	$$
\end{lemma}

\section{Proof of Theorem \ref{globale}}\label{main}

This section is devoted to the proof of the global existence of solutions for the problem  \eqref{sheatG}-\eqref{sheatG0} in the space $X$ that will be defined below. As it is claimed in the introduction, we will take the initial data in $L^{(p,\infty)}(\R^{N+2k})$, with $ p=\frac{N+2k}{2}(\rho -1).$

\begin{definition}\label{espX}
	We define the space $X$, formed by the functions $u: (0,\infty) \to L^{(p,\infty)}$, such that
	\begin{equation*}
		u \in L^{\infty}((0,\infty);L^{(p,\infty)}) ,
	\end{equation*}
	with the norm $\|u\|_{X}=\sup_{t>0} \|u(t)\|_{(p,\infty)}$.
\end{definition}

Considering this space and the Duhamel principle applied to \eqref{sheatG}--\eqref{sheatG0}, we define mild solutions below.
\begin{definition}\label{solbranda}
	For $u_0 \in L^{(p,\infty) }$, a global mild solution of the initial value problem $(\ref{sheatG})-(\ref{sheatG0})$ is a solution $u \in X$ of the integral equation
	\begin{equation}\label{branda}
		u(t) = S_\mathcal{G}(t)u_0 + \int_{0}^{t} S_\mathcal{G}(t-s) |u|^{\rho-1}u(s) \ ds := S_\mathcal{G}(t)u_0+ B(u)(t) ,
	\end{equation}
	such that $u(t)\rightarrow u_0$, as $t\rightarrow0^+$, in the sense of the distributions.
\end{definition}

From Proposition \ref{holder} and Remark \ref{ob}, we have the following simple and useful Lipschitz property.
\begin{lemma}
	Let $u,v \in X$ and $f$ given by $f(w)=|w|^{\rho-1}w, w\in\R$, for $1<\rho<r<\infty$. Then,
	\begin{equation}\label{desif}
		\||f(u)-f(v)|\|_{(\frac{r}{\rho},\infty)}\leq \rho\| u-v\|_{(r,\infty)}\left[\|u\|^{\rho-1}_{(r,\infty)}+\|v\|^{\rho-1}_{(r,\infty)}\right].
	\end{equation}
\end{lemma}

\begin{proof}[Proof of Theorem \ref{globale} (a).]
	We recall \eqref{branda} and define $\Psi: X \rightarrow X$ by $$\Psi(u)(t)=S_\mathcal{G}(t)u_0+B(u)(t) .$$ 
	Let $B_{X}(R)$ be the closed ball for radius $R>0$ as in Lemma \ref{fixedpoint}, and centered at the origin: $B_{X}(R):=\{w\in X:\|w\|_X\leq R\}$. We will apply Lemma \ref{fixedpoint}.

	Initially, we see that it follows from Proposition \ref{lpwest} that there is $C'>0$ such that
	\begin{equation*}
		\|S_\mathcal{G}(t) u_0\|_{(p,\infty)}\leq C' \|u_0\|_{(p,\infty)} ,
	\end{equation*}
	for all $t>0$. For a sufficiently small $\delta_p>0$, we conclude that 
	\begin{equation}\label{desiB}
		\|S_\mathcal{G}(t) u_0\|_{(p,\infty)}\leq \varepsilon ,
	\end{equation}
	where $0<\varepsilon<R$. 
	
	Now, as in \cite{FVR}, we define $h(s,\cdot) = u(t-s,\cdot)$, if $0\leq s\leq t$, and $h(s,\cdot)=0$ otherwise. Then, $B(u)=\zeta(h)$. Lemma \ref{nonl} then yields
	\begin{equation}
		\|B(u)(t)\|_{(p,\infty)} \leq  K_p \left(\|u\|_{(p,\infty)}\right)^{\rho} ,
	\end{equation}
	for some $K_p>0$. Before using Lemma \eqref{fixedpoint}, we recall that since $f:\R\to\R$ is a Borel measurable function and $u$ is measurable by assumption, then $t\mapsto f(u(t))$ also is a measurable function. Hence, the measurability of $t\mapsto B(u(t))$ follows from the proof of the vectorial Young inequality (see e.g. \cite{Arendt-B-Hi-N-01}). At this point, notice that $R>0$ chosen as in Lemma \ref{fixedpoint} depends on this $K_p>0$.

	Analogously, for $u, v \in B_{X}(R)$, we can use \eqref{desif} and Lemma \ref{nonl} to estimate
	\begin{equation}\label{psicont1}
		\|(B(u) - B(v))(t)\|_{(p,\infty)} \leq K \sup _{t>0} \left(\| u-v\|_{(p,\infty)} \left[\|u\|^{\rho-1}_{(p,\infty)}+\|v\|^{\rho-1}_{(p,\infty)}\right]\right).
	\end{equation}
	We are ready to apply Lemma \ref{fixedpoint}, which gives us a unique $u\in B_X(2\varepsilon)$ that is a solution of the integral equation $u=\Psi(u)$.

	It remains to show that the solution $u(t)\rightarrow u_0$, as $t\rightarrow 0^+$, in the sense of the distributions. Let $\varphi\in C^\infty_0(\R^{N+k})$. Since $C^\infty_0 \subset L^{(p',1)}$ (not densely injected), then
	\begin{equation*}
		|\langle S_\mathcal{G}(t)u_0-u_0,\varphi\rangle| =|\langle u_0,S_\mathcal{G}(t)\varphi-\varphi\rangle|
		\leq \|u_0\|_{(p,\infty)}\|S_\mathcal{G}(t)\varphi-\varphi\|_{(p^{\prime},1)}\rightarrow 0 .
	\end{equation*}

	Similarly, by Fubini's theorem and Hölder's inequality 
	\begin{equation*}
		\begin{aligned}
			|\langle B(u),\varphi\rangle| & \leq \int^t_0 \||u|^{\rho-1}u\|_{(\frac{p}{\rho},\infty)}ds\| S_\mathcal{G}(t-s)\varphi\|_{(\frac{p}{p-\rho},1)}\\
			& \leq C t  \|\varphi\|_{\frac{p}{p-\rho},1} \sup_{t>0} \|u\|_r^{\rho}\
			\rightarrow 0,
		\end{aligned}
	\end{equation*}
	as $t\rightarrow 0^+$. Therefore, $u(t)\rightarrow u_0$, as $t\rightarrow 0^+$, in the sense of the distributions.
	
	Furthermore, the last inequality Lemma \ref{fixedpoint} guarantees that, if $u,v\in$ are mild solutions of \eqref{sheatG}--\eqref{sheatG0} with initial conditions $u_0, v_0$, respectively, then 
	\begin{equation*}
		\|u(t) - v(t)\|_{(p,\infty)} \leq \frac{1}{1-2^\rho K_p \varepsilon^{\rho-1}} \|u_0 - v_0\|_{(p,\infty)} .
	\end{equation*}
\end{proof}

\begin{proof}[Proof of Theorem \ref{globale} (b).]
	Let $u_0\in L^{(p,\infty)}\cap L^{(q,\infty)}$ implies $u \in L^{\infty}(0,\infty; L^{(q,\infty)})$. Define Picard's sequence
	$$u_1=S_{\mathcal{G}}(t)u_0, \ \  \mbox{and}\ \ u_{k+1}=u_1+B(u_k).$$ 
	We know that $u_k\rightarrow u$ in $L^{\infty}(0,\infty; L^{(p,\infty)})$, 
	$$\|u_1\|_{(q,\infty)}\leq \|u_0\|_{(q,\infty)} ,$$ 
	and 
	$$\|u_{k+1}\|_{(q,\infty)}\leq C\|u_0\|_{(q,\infty)}+\|B(u_k)\|_{(q,\infty)} . $$
	By part (a), we have that, for all $k\in\N$,
	$$\|u_{k+1}\|_{(p,\infty)}\leq 2\varepsilon. $$
	From \eqref{nonl2}, 
	\begin{eqnarray*}
		\|B(u_{k+1})\|_{(q,\infty)}\leq K_p\sup_{t>0} \|u_k(t)\|^{\rho -1}_{(p,\infty)}\sup_{t>0} \|u_k(t)\|_{(q,\infty)} \leq K_p(2\varepsilon)^{\rho-1} \|u_k(t)\|_{(q,\infty)} .
	\end{eqnarray*}
	for $\|u_0\|_{(p,\infty)}<\delta_p$. For a possible smaller $\delta_p$, we have $a:=K_p(2\varepsilon)^{\rho-1}<1$. Then
	\begin{eqnarray}
		\|u_{k+1}\|_{(q,\infty)} &\leq& C\delta_q+K_p^{\rho-1}\|u_k\|_{(q,\infty)} \nonumber\\
		&\leq& C\delta_q(1+a+a^2+\cdots +a^k)\\
		&=& \frac{C \delta_q}{1-a} . \label{bounduk}
	\end{eqnarray}
	Let us see that $\{u_k\}$ is a Cauchy sequence in $L^{(q,\infty)}$. Define $v_k=u_{k+1}-u_k$, Then, \eqref{nonl2} again will yield
	\begin{eqnarray*}
		\|v_k\|_{(q,\infty)}&=&\|B(u_k)-B(u_{k-1})\|_{(q,\infty)}\\
		&=&\left\|\int^{\infty}_0S_{\mathcal{G}}(s)[f(h_k)-f(h_{k-1})]ds\right\|_{(q,\infty)}\\
		&\leq& K_p\sup_{t>0}\|f(u_k)-f(u_{k-1})\|_{(r,\infty)}\\
		&\leq& K_p \sup_{t>0}\left(\|u_k\|_{(p,\infty)}^{\rho-1}+\|u_{k-1}\|_{(p,\infty)}^{\rho-1}\right)\sup_{t>0}\|v_{k-1}\|_{(q,\infty)} .
	\end{eqnarray*}
	Now, if $\delta_q>0$ is sufficiently small, the above estimate and \eqref{bounduk} ensure that $\|v_k\|_{(q,\infty)} \rightarrow0$, as $k\rightarrow \infty$. Therefore, $\{u_k\}$ converges in $L^{(q,\infty)}$ to a function that equals $u$  by the uniqueness of the limit in the sense of distributions. This proves that	$u \in L^\infty((0,\infty); L^{(p,\infty)}(\R^{N+k}) \cap L^{(q,\infty)}(\R^{N+k}))$. 
	
	In this case, by interpolation, 
	\begin{equation}\label{uinLr}
		t\mapsto u(t) \in L^{\infty}(0,\infty,L^r(\mathbb{R}^{N+k}))
	\end{equation}
	with $p<r<q$, because setting $\eta$ such that $\frac{1}{r}=\frac{\eta}{p}+\frac{1-\eta}{q}$, $q_0,q_1=\infty$ and $s=r$, we have 
	\begin{equation*}
		(L^{(p,\infty)},(L^{(q,\infty)})_{(\eta,r)}=L^{(r,r)}=L^r .
	\end{equation*}

\end{proof}

\begin{proof}[Proof of Theorem \ref{globale} (c).]
	For the decay, it needs $1<p<q<\infty$,  $\sigma=\frac{N+2k}{2}\left(\frac{1}{p}-\frac{1}{r}\right)$, and the space
	$Y$, formed by the functions $u\in X$, such that
	\begin{equation*}\label{espX2}
		t\mapsto t^{\sigma}u(t) \in L^{\infty}((0,\infty);L^{(q,\infty)}) ,
	\end{equation*}
	endowed with the norm 
	\begin{equation}
		\|u\|_{Y}= \sup_{t>0}\|u(t)\|_{(p,\infty)} + \sup_{t>0}t^{\sigma}\|u(t)\|_{(q,\infty)}.
	\end{equation}

	Given the estimates obtained in the proof of part (a), to apply Lemma \ref{fixedpoint} with the Banach space $Y$, we only need to estimate $t^{\sigma}u(t)$ in $L^{(q,\infty)})$. From Proposition \ref{lpwest}, we have
	\begin{equation}\label{desiB2}
		\|S_\mathcal{G}(t) u_0\|_{(q,\infty)}\leq \varepsilon ,
	\end{equation}
	and \eqref{nonl2} yields
	\begin{equation*}
		\|B(u)\|_{(q,\infty)}\leq K_p\|u(t)\|_{(p,\infty)} \|u(t)\|_{(q,\infty)},
	\end{equation*}
	or
	\begin{equation*}
		\sup_{t>0} t^\sigma\|B(u)\|_{(q,\infty)}\leq K_p \sup_{t>0} \|u(t)\|_{(p,\infty)} \sup_{t>0}  t^\sigma\|u(t)\|_{(q,\infty)}^{\rho-1} \leq \|u\|_Y^\rho.
	\end{equation*}
	Then, one can apply Lemma \ref{fixedpoint} to find a mild solution of \eqref{sheatG}--\eqref{sheatG0} in the ball $B_Y(2\varepsilon)$. In particular, 
	\begin{equation*}
		\|u(t)\|_{(q,\infty)}\leq Ct^{-\sigma}, \ t>0.
	\end{equation*}

\end{proof}

\begin{proof}[Proof of Theorem \ref{globale} (d).] For the uniqueness, take two mild solutions $u,v \in L^\infty((0,\infty); L^{(p,\infty)}(\R^{N+k}) \cap L^{(q,\infty)}(\R^{N+k}))$ for the same Cauchy problem \eqref{sheatG}--\eqref{sheatG0}, with $1<p<q<\infty$. Then,
	\begin{eqnarray*}
		\|u(t) - v(t)\|_{(p,\infty)} & \leq & \int_0^t \| S_\mathcal{G}(t-s) [f(u(s)) - f(v(s))] \| _{(p,\infty)} ds \\
		& \leq & 2C \left(\sup_{t>0} \| u(t) \|_{(q,\infty)}\right)^{\rho-1} \\
		&& \times \int_0^t (t-s)^{-\frac{N+2k}{2}(\frac{1}{r} - \frac{1}{p})}  \|u(s) - v(s) \| _{(p,\infty)} ds ,
	\end{eqnarray*}
	with $\frac{1}{r} = \frac{\rho-1}{q} + \frac{1}{p}$. Then, $\frac{N+2k}{2}(\frac{1}{r} - \frac{1}{p})=\frac{p}{q}<1$. Define $\varphi(t) = \sup_{s\in(0,t)} \|u(s) - v(s)\|_{(p,\infty)}$. By the Singular Gronwall Lemma, $\varphi(t)=0$, for all $t>0$, which means that $u\equiv  v$ in $L^{(p,\infty)}(\R^{N+k})$.
	
\end{proof}

\section{Proof of Theorem \ref{self}} \label{proofsym}

We will split the proof of Theorem \ref{self} into small lemmas.

The Lorentz space $L^{(p,q)}(\R^{N+k})$ enjoys the following scaling relation:
\begin{equation}\label{norminv}
	\|f(\lambda \cdot, \lambda^2\cdot)\|_{(p,q)} = \lambda^{-\frac{N+2k}{p}}\|f\|_{(p,q)} .
\end{equation}
To see this, recall that $m(T(E)) = |\det T| m(E)$ for any Lebesgue measurable set $E$ and $T\in GL_n(\R)$. Thus, consider the dilation $\delta_\lambda(x,y) = (\lambda^{-1}x, \lambda^{-2}y)$, we conclude that 
$$m(\delta_\lambda(E))\leq t \iff m(E)\leq \lambda^{N+2k}t .$$ 
Using this in the rearrangement function and then in the averaging function (see Section \ref{pre}), we have 
$$ (f_\lambda)^*(t) = f^*(\lambda^{N+2k}t)$$
and
$$(f_\lambda)^*(t) = f^*(\lambda^{N+2k}t) ,$$
respectively, where $f_\lambda  = f(\lambda \cdot, \lambda^2\cdot)$. Now, this latter identity inside \eqref{norma} implies \eqref{norminv}.

Identity \eqref{norminv} says that, if we want the initial condition to comply with \eqref{scalingmap}, we must take it in the space $L^{(p,\infty)}(\R^{N+k})$, $p=\frac{N+2k}{2}(\rho-1)$, which is the case of Theorem \ref{globale}. Furthermore, the Marcinkiewicz space $L^{(p,\infty)}(\R^{N+k})$ is invariant by the \eqref{scalingmap}:
$$\|u_\lambda\|_{(p,q)} = \|u\|_{(p,q)} .$$

Consider the following sequence
\begin{equation}\label{picard}
	u_1(x,t)=S_\mathcal{G}(t)u_0(x), \text{ and } u_{k+1}(x,t)=u_1(x,t)+B(u_k), \ k\in\mathbb{N}.
\end{equation}
The proof of Theorem \ref{self} relies on the fact that Picard's sequence \ref{picard} propagates the stated symmetries and positivity of the initial conditions. Indeed, the Theorem will be proved if we show that if the initial condition is the symmetry (resp. is non-negative), then each term of Picard's sequence will also be invariant by that symmetry property (resp. is positive). Therefore, the Theorem will be proved after the lemmas below.
\qed

\begin{lemma}\label{scalor}
	\begin{enumerate}
		\item[(a)] If $\varphi$ is a nonnegative function, then $S_\mathcal{G}(t)\varphi$ is positive.
		\item[(b)] Let $\varphi \in L^{(p, \infty)}$ be a homogeneous function of degree $-\frac{2}{\rho-1}$ in the following sense:
		\begin{equation}\label{homog}
			\varphi(\lambda x, \lambda^2 y) = \lambda^{-\frac{2}{\rho-1}} \varphi(x,y),  \ (x,y)\in\R^{N+k} .
		\end{equation}
		Then $S_\mathcal{G}(t)\varphi$ is invariant by the scaling map \eqref{scalingmap}.
		\item[(c)] If the function $\varphi$ is invariant under the action of $\mathcal{A}$, then $S_\mathcal{G}(t)\varphi$ also is.
	\end{enumerate}
	\eqref{scalingmap}.
\end{lemma}

\begin{proof}
	The proof of (a) is immediate from the positivity of the heat kernel \eqref{HK}.
	
	To prove (b), we will use expression \eqref{HK}:
	\begin{eqnarray*}
		&& K(\lambda x, \lambda x_0, \lambda^2y,\lambda^2 t) \\
		&=& \frac{1}{(2\pi)^{\frac{N+2k}{2}}}\int_{\mathbb{R}^k}\left(\frac{|\xi|}{\sinh(|\xi|\lambda^2t)}\right)^{\frac{N}{2}} e^{i\xi  \cdot \lambda^2 y- \frac{|\xi|}{2}\left((|\lambda x|^2 + |\lambda x_0|^2)  \coth(|\xi| \lambda^2 t) -2 (\lambda x)\cdot (\lambda x_0) \csch(|\xi|\lambda^2 t) \right)} \di\xi\\
		&=& \frac{\lambda^{-2k-N}}{(2\pi)^{\frac{N+2k}{2}}}\int_{\mathbb{R}^k}\left(\frac{|\xi|}{\sinh(|\xi|t)}\right)^{\frac{N}{2}} e^{i\xi  \cdot y- \frac{|\xi|}{2}\left((|x|^2 + |x_0|^2)  \coth(|\xi| t) -2x\cdot x_0 \csch(|\xi|t) \right)} \di\xi\\
		&=&\lambda^{-2k-N}K(x,x_0,y,t)
	\end{eqnarray*}
	Then, a change of variables in \eqref{solinear}, and the homogeneity \eqref{homog} yield
	\begin{eqnarray*}
		(S_\mathcal{G}(\cdot)\varphi)_\lambda(x,y,t) & = &  \lambda^{\frac{2}{\rho-1}} \int_{\R^{N+k}} K(\lambda x,w,z ,\lambda^2 t) \ \varphi(w, \lambda^2 y - z)  \di w \di z \\
		& = &  \lambda^{N+2k+\frac{2}{\rho-1}} \int_{\R^{N+k}} K(\lambda x,\lambda w, \lambda^2z ,\lambda^2 t) \ \varphi(\lambda w, \lambda^2 y - \lambda^2z)  \di w \di z \\
		& = & \left(S_\mathcal{G}(\cdot)\varphi \right)(x,y,t).
	\end{eqnarray*}
	
	The proof of (c) relies on the fact that $T\in\mathcal{A}$ is written as $T=(T_1,T_2)$ for orthogonal matrices $T_1\in \mathcal{O}(N)$ and $T_2\in\mathcal{O}(k)$. Then $|T_1 x|=|x|$ and $(T_2^t)^{-1}=T_2$. Hence, the change of variables $T_2^t \xi = \xi^\prime$ gives
	\begin{eqnarray*}
		K(x_0,T(x,y),t) & = & \frac{1}{(2\pi)^{\frac{N+2k}{2}}} \int_{\mathbb{R}^k} \left(\frac{|\xi|}{\sinh(|\xi|t)}\right)^{\frac{N}{2}} e^{i\xi  \cdot T_2 y- \frac{|\xi|}{2}\left((|T_1x|^2 + |x_0|^2)  \coth(|\xi| t) -2 T_1x \cdot x_0 \csch(|\xi|t) \right)} \di\xi\\
		& = & K(x, T_1^t x_0,y,t) .
	\end{eqnarray*}
	Then, the change $(w,z) = T(w,z)$ gives $S_\mathcal{G}(t)\varphi(T(x,y)) = S_\mathcal{G}(t)\varphi(x,y)$.
\end{proof}

\begin{lemma}\label{scalb}
	Let $u \in L^\infty(0,\infty; L^{(p, \infty)})$.
	\begin{enumerate}
		\item[(a)] If $u$ is positive, then $B(u)$ is also positive.
		\item[(b)] If $u$ is invariant by the scaling map \eqref{scalingmap}, then so is $B(u)$.
		\item[(c)] If $u$ is invariant under the action of $\mathcal{A}$, then so is $B(u)$.
	\end{enumerate} 
\end{lemma}

\begin{proof}
	The positivity in (a) follows immediately from $f(u)>0$ and Lemma \ref{scalor} (a).
	Notice that, if $\varphi$ is invariant under the dilation \eqref{scalor}, then
	\begin{equation}\label{homop}
		\lambda^{\frac{2\rho}{\rho-1}} \varphi^{\rho}(\lambda x,\lambda^2y) = (\varphi_{\lambda}(x,y))^\rho = \varphi^\rho(x,y) .
	\end{equation}
	Hence,
	
	\begin{eqnarray*}
		[S_\mathcal{G}(\lambda^2t)\varphi^{\rho}](\lambda x, \lambda^2y) & = &   \int_{\mathbb{R}^{N+k}} K(\lambda x,w,\lambda^2y-z,\lambda^2t) \varphi^{\rho}(w,z) \di w \di z\\
		&=& \lambda^{ N +2k } \int_{\mathbb{R}^{N+k}} K(\lambda x, \lambda w,\lambda^2 (y-z),\lambda^2t) \varphi^{\rho}(\lambda w, \lambda^2 z) \di w \di z\\
		&=&  \int_{\mathbb{R}^{N+k}} K( x,  w, y-z,t) \varphi^{\rho}(\lambda w, \lambda^2 z) \di w \di z\\
		&=& \lambda^{ - \frac{2\rho}{\rho-1}} S_\mathcal{G}(t)\varphi^{\rho}(x,y)
	\end{eqnarray*}
	Then,
	\begin{eqnarray*}
		(B(u))_\lambda &=& \left[ \int^{\cdot}_0S_\mathcal{G}(\cdot-s)u^{\rho}(s)ds\right]_{\lambda} \\
		&=& \lambda^{\frac{2}{\rho-1}}\int^{\lambda^2t}_0S_\mathcal{G}(\lambda^2t-s)u^{\rho}(\lambda x, \lambda^2y, s) ds\\
		&=&\lambda^{2+\frac{2}{\rho-1}} \int^t_0 S_\mathcal{G}(\lambda^2(t-s)) u^{\rho}(\lambda x,\lambda^2y, \lambda^2s)ds\\
		&=&\lambda^{2+\frac{2}{\rho-1}- \frac{2\rho}{\rho-1}} \int^t_0S_\mathcal{G}(t-s)u^{\rho}(s)ds\\
		&=&\int^t_0S_\mathcal{G}(t-s)u^{\rho}(s)ds \\
		&=& B(u).
	\end{eqnarray*}
	It proves (b). The invariance of $f(u)$ by $\mathcal{A}$ and Lemma \ref{scalor} (c) finish the proof.
\end{proof}

\section{Final remarks}

We finish the paper with some remarks that are either useful to give a more concrete taste to our results or to give insights to further related research.

\begin{enumerate}
	\item[(a)] Let $T_1$ and $T_2$ be the rotation matrices in $\R^N$ and $\R^k$, respectively. Then, $T=(T_1,T_2)\in\mathcal{A}$. Then , if $u_0(x,y) = u_0(|x|,|y|)$, Theorem \ref{self} guarantees that the solutions are \textit{cylindrical} in the sense that $u(x,y,t) = u(|x|,|y|,t) = u_0(x,y) = u_0(|x|,|y|)$, for all $(x,y,t)\in \R^{N+k}\times(0,\infty)$. 
	
	\item[(b)] If $u_0(x,y) = \varepsilon|x|^{-\frac{2}{\rho-1}} |y|^{-\frac{1}{\rho-1}}$, then, for $\varepsilon>0$ sufficiently small, Theorem \ref{self} gives the existence of a positive, self-similar, and symmetric solution of \eqref{sheatG}--\eqref{sheatG0}. Such a solution belongs to $L^\infty((0,\infty);L^{(p,\infty)})$.

	\item[(c)] Given a initial datum $u_0 \in L^{(p,\infty)}$ satisfying \eqref{homog}, the self-similar solution $u$ given by Th. \ref{self} is stable in the sense of Eq. \eqref{contdep} in Th. \ref{globale}. Nevertheless, the proof of \eqref{contdep} tells us that it can be replaced with
	\begin{equation}
		\|u(t) - v(t)\|_{(p,\infty)} \leq M \|S_{\mathcal{G}}(t)\left( u_0 - v_0 \right) \|_{(p,\infty)} .
	\end{equation}
	Therefore, if in addition 
	$$\|S_{\mathcal{G}}(t)\left( u_0 - v_0 \right) \|_{(p,\infty)} \rightarrow 0,$$
	as $t\rightarrow\infty$, then 
	\begin{equation}
		\lim_{t\rightarrow\infty} \|u(t) - v(t)\|_{(p,\infty)} =0,
	\end{equation}
	that is, each self-similar solution is a basin of attraction. Similar results occur with the heat equation with the fractional Laplacian \cite{FVR} and the Navier-Stokes equation \cite{Can-Kar-05}.
	
	\item[(d)] If a positive classical solution is self-similar and we make $\lambda = t^{-\frac{1}{2}}$ in \eqref{scalingmap}, then $u(x,y,t) = t^{-\frac{1}{\rho-1}} v(\xi, \eta)$, where $\xi = t^{-\frac{1}{2}} x$ and $\eta = t^{-1}y$. Straightforward computations show that $v$ must be a solution of
	\begin{equation}\label{selfstat}
		\Delta_{\mathcal{G}}v - \nabla_{(\xi,\eta)} v \cdot \left(\frac{\xi}{2},\eta\right) - \frac{1}{\rho-1} v + v^\rho .
	\end{equation}
	$\Delta_{\mathcal{G}}$ in \eqref{selfstat} is the operator defined in \eqref{Grushiop} in variables $(\xi,\eta)$ and $\nabla_{(\xi,\eta)}$ is the Euclidean gradient. As
	\begin{equation*}
		\Delta_z v - \nabla_z v \cdot z - \frac{1}{\rho-1} v + v^\rho 
	\end{equation*}
	raised self-interest for being related to positive forward self-similar solutions of $u_t = \Delta u + u^\rho$ (see \cite{Giga86} and references thereof), also \eqref{selfstat} is interesting itself.
	
	\item[(e)] 
	The last comment concerns the nonexistence of global solutions for \eqref{sheatG}--\eqref{sheatG0}. For this, define the following energy functional associated with it:
	\begin{equation}
		E(u)=\frac{1}{2} \int_{\R^{N+k}} |\nabla_{\mathcal{G}} u|^2 \ dx - \frac{1}{p+1} \int_{\R^{N+k}}|u|^{p+1} d x .
	\end{equation}
	Now, define the Sobolev space
	$$H^1_\mathcal{G} = \{u\in L^2(\R^{N+k}): \nabla_{\mathcal{G}}u\in L^2(\R^{N+k})\} ,$$
	endowed with the norm
	$$\| u \| = \|u\|_{L^2} + \|\nabla_{\mathcal{G}}u\|_{L^2} .$$
	Here $\nabla_{\mathcal{G}} = (\nabla_x, |x|\nabla_y)$ denotes the intrinsic gradient of $\Delta_{\mathcal{G}}$.  Let $u_0 \in L^{\infty} \cap H_\mathcal{G}^1$ be such that $E(u_0)<0$. If the solution $u$ belongs to $C((0,T);H_\mathcal{G}^1)$, then we can apply the following Green identity. Let $B_1(R)$ and $B_2(R)$ be balls of radius $R>0$ in $\R^N$ and $R^k$, respectively. For $B(R) = B_1(R)\times B_2(R)\subset \R^{N+k}$, we have
	\begin{equation*}
		\int_{B(R)} u\Delta_{\mathcal{G}}v = \int_{B_2(R)}\int_{S_1(R)} u\frac{\partial v}{\nu_x} + \int_{B_1(R)}\int_{S_2(R)} v\frac{\partial v}{\nu_x} 
		-\int_{B(R)} \nabla_{\mathcal{G}}u\cdot \nabla_{\mathcal{G}} v .
	\end{equation*}
	In particular, for $v=u \in C_0^\infty$, for $R>0$, we get
	\begin{equation*}
		\int_{B(R)} u\Delta_{\mathcal{G}}v = -\int_{B(R)} |\nabla_{\mathcal{G}}u|^2 .
	\end{equation*}
	By density, those identities should work with $v,u\in H_\mathcal{G}^1$. Hence, if the solution $u\in C((0,T);H_\mathcal{G}^1)\cap L^{\rho+1}$, then we can prove that $E(u(\cdot)) \in C([0,T])\cap C^1((0,T))$ and 
	\begin{equation*}
		\frac{d}{dt}E(u(t)) = - \int_{\R^{N+k}} u_t^2 .
	\end{equation*}
	This can be proved by following \cite[Lemma 17.5]{Soup-Quit-07} with the Green identities above. From this on, the Levine concavity method \cite{Levine73} can be applied to show that  $T:=T_{\max }\left(u_0\right)<\infty$.

\end{enumerate}

\subsection*{Acknowledgment}

Most of this research was carried out while Kogoj and Viana visited the Abdus Salam International Centre for Theoretical Physics. We want to thank the ICTP for the hospitality during their short visit and INdAM for the support. A. Viana is partially supported by CNPq under the grant number 308080/2021-1. A.E. Kogoj has been partially supported by the Gruppo Nazionale per l'Analisi Matematica, la Probabilit\`a e le loro Applicazioni (GNAMPA) of the Istituto Nazionale di Alta Matematica (INdAM). 

We thank Giulio Tralli for many valuable discussions.  We thank the Referee for carefully reading the manuscript
and for her/his remarks that greatly improved the presentation of our results.

\vspace{12pt}\noindent
{\bf Conflict of interest statement.}
The authors declare that they have no conflicts of interest.

\bibliographystyle{abbrv}
\bibliography{ref}
\end{document}